\documentclass[11pt]{article}
 \usepackage[T1]{fontenc}
 \usepackage{chngcntr}
\usepackage{amsthm}
\usepackage{amssymb}
\usepackage{amsmath}
\usepackage{comment}
\usepackage{thm-restate}
\usepackage{latexsym}
\usepackage{graphicx,enumerate,color}
\usepackage{url} 
\usepackage[affil-it]{authblk}
\usepackage{hyperref}
\usepackage[noabbrev,capitalise]{cleveref}

\usepackage{tikz}
\usetikzlibrary{positioning}

\usepackage{color}
\usepackage[normalem]{ulem}

\usepackage[margin=1in]{geometry}

\theoremstyle{plain}
\newtheorem{thm}{Theorem}[section]

\newtheorem{conj}[thm]{Conjecture}

{\noindent \emph{Proof.} {}{#1}{}}{\hfill
	$\Diamond$\vspace{1em}}

\theoremstyle{plain} 
\newcommand{\thistheoremname}{}
\newtheorem{genericthm}[section]{\thistheoremname}

\theoremstyle{definition}

\newcounter{counter}
\newcommand{\less}{\setminus}
 \def\es{\emptyset}
 \def\dfn#1{{\sl #1}}
 \newcommand{\ceil}[1]{{\left\lceil #1 \right\rceil}}

 \title{Dominating Hadwiger's Conjecture holds for all $2K_2$-free  graphs}
 \author{Zi-Xia Song\thanks{Supported by  NSF grant    DMS-2153945. E-mail address: {\tt  Zixia.Song@ucf.edu}}\, \,   and Thomas Tibbetts\thanks{Supported by  NSF grant    DMS-2153945 supplemental funding  for  undergraduate students  at the University of Central Florida. E-mail address: {\tt  Thomas.Tibbetts@ucf.edu}}}

  \affil{  \small {Department  of Mathematics, University of Central Florida, Orlando, FL 32816, USA}}
  
\date{}
\begin{document}

\maketitle

\begin{abstract}
 A   dominating $K_t$ minor in a graph $G$ is a sequence $(T_1,\dots,T_t)$ of pairwise disjoint non-empty connected subgraphs of $G$, such that for $1 \leq i<j\leq t$, every vertex in $T_j$ has a neighbor in $T_i$. Replacing ``every vertex in $T_j$'' by ``some vertex in $T_j$'' retrieves the standard definition of a $K_t$ minor.  The strengthened notion was introduced  by   Illingworth and Wood [arXiv:2405.14299], who   asked whether    every   graph with chromatic number $t$  contains a dominating $K_t$ minor.   This   is   a  substantial strengthening of  the celebrated  Hadwiger’s Conjecture,  which asserts  that every   graph with chromatic number $t$  contains a   $K_t$ minor.  At the ``New Perspectives in Colouring and Structure'' workshop  held at   the Banff International Research Station   from September 29 - October 4, 2024,   Norin  referred to this  question as the ``Dominating Hadwiger's Conjecture'' and     believes it is likely false. 
 In this paper we    prove that  the Dominating Hadwiger's Conjecture holds for  all   $2K_2$-free  graphs.  A key component of our proof is the clever use of the existence of an induced banner,  obtained by adding a  vertex adjacent to exactly one vertex on  a cycle of length four.
\end{abstract}

\section{Introduction}
All graphs in this paper are finite and simple. For a graph $G$, we use   $\alpha(G)$ and $\chi(G)$  to denote the     independence number and chromatic number  of $G$, respectively.   A \dfn{clique} in a graph $G$ is a set   of pairwise adjacent vertices. For a vertex $v \in V(G)$, we use $N(v)$ to denote the set of vertices in $G$ which are all adjacent to $v$.   If $X,Y \subseteq V(G)$ are disjoint, we say that $X$ is \dfn{complete} to $Y$ if  every vertex in $X$ is adjacent to all vertices in $Y$, and $X$ is \dfn{anticomplete} to $Y$ if no vertex in $X$ is adjacent to any vertex in $Y$. If $X = \{x\}$, we simply say $x$ is complete to $Y$ or $x$ is anticomplete to $Y$. 
 The subgraph of $G$ induced by $X$ is  denoted as $G[X]$. We denote by $X \less Y$ the set $X - Y$  and $G \less X$ the subgraph of $G$ induced on $V(G) \less X$.     If $X = \{ x\}$, we simply write $X \setminus x$ and $ G \setminus x$, respectively. Throughout, we use the notation ``$A :=$'' to indicate that $A$ is defined to be the expression on the right-hand side. For any positive integer $k$, we write $[k]:=\{1, \dots , k\}$. \medskip
 
 Given a graph $H$, we say that a graph $G$ is \dfn{$H$-free} if $G$ has no induced subgraph isomorphic to $H$.   A graph  $G$ contains a  \dfn{$K_t$ minor} if  there exist pairwise disjoint non-empty connected subgraphs $T_1, \ldots, T_t$ of $G$, such that for $1 \leq i<j\leq t$, some vertex in $T_j$ has a neighbor in $T_i$.  Our work is motivated by the  celebrated Hadwiger's Conjecture~\cite{Had43}.

\begin{conj}[Hadwiger's Conjecture~\cite{Had43}]\label{c:HC} For every integer $t \geq 1$, every graph with no $K_t$ minor is $(t-1)$-colorable. 
\end{conj}

Hadwiger's Conjecture is widely considered among the most important problems in graph theory and has motivated numerous developments in graph coloring and graph minor theory.  We refer the reader to   recent surveys  \cite{Toft, Seymoursurvey, Norinsurvey} for further  background on Hadwiger's Conjecture. 
Hadwiger~\cite{Had43} and Dirac~\cite{Dirac52} independently showed that \cref{c:HC}  holds for $t \leq 4$. However, for $t\ge5$, Hadwiger's Conjecture implies the Four Color Theorem~\cite{AH77,AHK77}.   Wagner~\cite{Wagner37} proved that the case $t=5$ of Hadwiger's Conjecture is, in fact, equivalent to the Four Color Theorem, and the same was shown for $t=6$ by Robertson, Seymour and  Thomas~\cite{RST93}. Despite receiving considerable attention over the years, Hadwiger's Conjecture remains  open for $t\ge 7$.  In the 1980s, Kostochka~\cite{Kostochka82,Kostochka84}  and Thomason~\cite{Thomason84} proved that every   graph with chromatic number $t$   contains a clique minor of order  $ \Omega(t/\sqrt{\log t})$,  which was improved to $ \Omega(t/(\log t)^{1/4 + \epsilon})$ by Norin, Postle, and the first author~\cite{NPS23}, and very recently, to $ \Omega(t/\log\log t)$ by Delcourt and Postle \cite{DP25}.   \medskip

  Recently, Illingworth and Wood~\cite{IW24}  introduced the notion of a  ``dominating $K_t$  minor'': a graph $G$ contains a \dfn{dominating $K_t$ minor} if there exists a  sequence $(T_1,\dots,T_t)$ of pairwise disjoint non-empty connected subgraphs of $G$, such that for $1 \leq i<j\leq t$,  every vertex in $T_j$ has a neighbor in $T_i$. Replacing ``every vertex in $T_j$'' by ``some vertex in $T_j$'' recovers the standard definition of a $K_t$ minor  of $G$. 
 Illingworth and Wood explored in what sense dominating $K_t$ minors  behave like (non-dominating) $K_t$ minors: they observed   that the two notions are equivalent for $t\leq 3$, but are already very different for $t=4$, for instance, the $1$-subdivision of $K_n$ contains a $K_n$ minor for all $n \geq 4$, yet lacks a dominating $K_4$ minor. This led them to propose a natural strengthening of Hadwiger’s Conjecture:  \emph{is every graph  with no dominating $K_t$ minor  $(t-1)$-colorable?}  Norin~\cite{Norin24} referred to this as the ``Dominating Hadwiger's Conjecture'' and     suspects the conjecture is false.

\begin{conj}[Dominating Hadwiger's Conjecture~\cite{IW24}]\label{c:dHC}
    For every integer $t \geq 1$, every graph with no dominating $K_t$ minor is $(t-1)$-colorable.
\end{conj}
 
Conjecture~\ref{c:dHC} is  trivially true for $t\le3$.  Illingworth and Wood~\cite{IW24} verified the case $t=4$ by  showing  that every graph with no dominating $K_4$ minor  is $2$-degenerate and $3$-colorable. Norin~\cite{Norin24} announced a proof that every graph with no dominating $K_5$ minor  is $5$-degenerate.   More generally, Illingworth and Wood~\cite{IW24}  proved that 
every graph with no dominating $K_t$ minor  is $2^{t-2}$-colorable. 
  They further  provided two pieces of evidence for the Dominating Hadwiger's  
Conjecture: (1) It is true for almost every graph, (2) Every graph $G$ with no dominating $K_t$ minor has a $(t-1)$-colorable induced subgraph on at least half the vertices, which implies there is an independent set of size at least $\ceil{|V(G)|/(2t-2)}$.  \medskip

The \dfn{dominating Hadwiger number} $h_d(G)$ of a graph $G$ is the largest integer $t$ such that $G$ contains a dominating $K_t$ minor.  In our  search for   potential counterexamples, we prove that the Dominating Hadwiger's Conjecture holds for all $2K_2$-free graphs, where $2K_2$ denotes the complement of the cycle of length four.

\begin{restatable}{thm}{dominating}\label{t:main}
Every $2K_2$-free graph $G$ satisfies  $h_d(G) \geq \chi(G)$.
 \end{restatable}

We prove Theorem~\ref{t:main} in Section~\ref{s:main}. Our novel proof  cleverly   relies on the existence of an induced \dfn{banner},  obtained by adding a  vertex adjacent to exactly one vertex on  a cycle of length four. It also uses  Theorem~\ref{t:forbidhole}  that follows directly from the proof of  Theorem 2.2 of  B. Thomas and the first author~\cite{ST17}. For brevity, we omit its proof here.

\begin{thm}\label{t:forbidhole} Every   $\{C_4, C_5, C_6, \ldots, C_{2\alpha(G)}\}$-free graph $G$ satisfies   $h_d(G) \geq \chi(G)$.
\end{thm}

Theorem~\ref{t:main} was used by Scully and the first author~\cite{SS25} to provider further evidence for  the Dominating Hadwiger's Conjecture   in the case of  graphs $G$ with $\alpha(G)\le 2$. A significantly stronger result than Theorem~\ref{t:alpha=2} was obtained in~\cite{SS25}.

\begin{thm}[Scully and Song~\cite{SS25}]\label{t:alpha=2}
 Let $G$ be a graph on $n$ vertices with $\alpha(G)\le 2$. Let $H\ne K_2\cup K_3$ be a graph with $|V(H)|\le5$ and $\alpha(H)\le2$. If $G$ is $H$-free, then $h_d(G) \geq \chi(G)$.
 \end{thm}

It is worth noting that Micu~\cite{Micu05}   proved that Hadwiger's Conjecture holds for all $2K_2$-free graphs. For completeness, we recall the proof here, as it is remarkably short and relies on the Strong Perfect Graph Theorem~\cite{SPGT}.   However, Micu's argument does not extend to establishing a dominating $K_{\chi(G)}$ minor.

\begin{thm}[Micu~\cite{Micu05}]\label{t:mainHC}
 Every $2K_2$-free graph $G$  contains a $K_{\chi(G)}$ minor.
\end{thm}
\begin{proof} Let $G$ be a $2K_2$-free graph. Suppose $G$ contains  no $K_{\chi(G)}$ minor. We choose $G$ with $|V(G)|$ minimum. Then  $G$ is not perfect. By the Strong Perfect Graph Theorem~\cite{SPGT},  $G$ contains an induced path of length three,  say with vertices $v_1, v_2, v_3, v_4$ in order. Let 
\[A:=\{v\in V(G)\less \{v_1, v_2, v_3, v_4\}\mid v \text{ is anticomplete to } \{v_1, v_2\}\} \text{ and }\]
\[B:=\{v\in V(G)\less (\{v_1, v_2, v_3, v_4\}\cup A)\mid v \text{ is anticomplete to } \{v_3, v_4\}\}.\]
Since $G$ is $2K_2$-free, we see that both $A$ and $B$ are independent sets in $G$, $v_4$ is anticomplete to $A$ and $v_1$ is anticomplete to $B$. Let $C:= \{v_1, v_2, v_3, v_4\}\cup A\cup B$. By the choice of $A$ and $B$,  each vertex in $ V(G)\less C$ is adjacent to at least one vertex in both  $\{v_1, v_2\}$  and  $\{v_3, v_4\}$. Moreover,  $A\cup\{v_2, v_4\}$ and $B\cup\{v_1, v_3\}$ are independent  sets in $G$ and so $ G[C]$  is $2$-colorable.  Thus $\chi(G\less C)\ge \chi(G)-2$. By the minimality of $G$, $G\less C$ contains a $K_{\chi(G)-2}$ minor. This, together with $G[\{v_1, v_2\}]$ and $G[\{v_3, v_4\}]$, yields a $K_{\chi(G)}$ minor in $G$, a contradiction.
\end{proof}

\section{Proof of Theorem~\ref{t:main}}\label{s:main}

A \dfn{banner} is the graph obtained from the cycle  $C_4$ by adding a new vertex    adjacent to exactly one vertex on $C_4$. We use $B=(b_1, b_2, b_3, b;b')$ to denote a banner, where $V(B)=\{b_1, b_2, b_3, b, b'\}$ and $E(B)=\{b_1b_2, b_2b_3, b_3b, bb_1, bb'\}$.   For instance, $T[\{b_1, b_2, b_3, b, b'\}]=B$ is a banner, where $T$ is depicted in Figure~\ref{fig:banner}.  \medskip

\begin{figure}[h]
    \centering
   \begin{tikzpicture}[scale=2]
		\foreach \i in {1,...,3}
		\node[circle, fill = black, inner sep = 1.6pt, label=above:$b_{\i}$] (v\i) at ({90-72*(\i-2)}:1)  {};
		\foreach \i in {4,5}
		\node[circle, fill = black, inner sep = 1.6pt, label=below:$b_{\i}$] (v\i) at ({90-72*(\i-2)}:1)  {};
		
		\node[circle, fill = black, inner sep = 1.6pt, label=below:$b$] (b) at (0,0) {};
		\node[circle, fill = black, inner sep = 1.6pt, label=above:$b'$] (b') at (0,.5) {};
		\draw (v1)--(v2)--(v3)--(v4)--(v5)--(v1);
		
		\foreach \i in {1,3,4,5}
		\draw (b) -- (v\i);
		\draw (b) -- (b');
	\end{tikzpicture}
    \caption{The graph  $T$  in which  $T[\{b_1, b_2, b_3, b, b'\}]=B$ is a banner.}
    \label{fig:banner}
\end{figure}
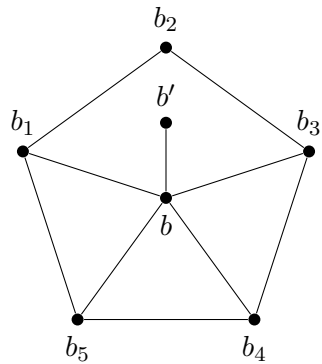

In this section we prove Theorem~\ref{t:main} which we restate below for convenience. \dominating*

\begin{proof}
 Suppose the statement is false.    Let $G$ be a counterexample   such that $|V(G)|$ is minimum. Then $G$ is connected and $h_d(G) < \chi(G)$. We next prove several claims. \\

\setcounter{counter}{0}

\noindent {\bf Claim\refstepcounter{counter}\label{banner}  \arabic{counter}.}
 If $G$ contains an induced banner $B=(b_1, b_2, b_3, b; b')$,  then there exist two adjacent   vertices $b_4, b_5\in V(G)$ such that $b_4$ is complete to     $\{b, b_3\} $  and anticomplete to   $\{b', b_1, b_2\} $;  $b_5$ is complete to     $\{b, b_1\} $  and anticomplete to   $\{b', b_2, b_3\} $. That is, $G$ contains an induced $T$ such that $T$ contains $B$ as an induced subgraph, where $T$ is depicted in Figure~\ref{fig:banner}. 
 
 \begin{proof} Suppose $G$ contains an induced   banner $B=(b_1, b_2, b_3, b; b')$. 
Let $D_1 :=\{b_1, b,b'\}]$ and $D_2 := \{  b_2,b_3 \}$.  Let 
\[A_1:=\{v\in V(G) \less V(B)\mid v \text{ is anticomplete to } D_1\} \text{ and}\]
\[ A_2:=\{v\in V(G) \less (V(B)\cup A_1)\mid v \text{ is anticomplete to } D_2\}.\]
 Since $G$ is $2K_2$-free, we see that $b_2$ is anticomplete to $A_1$ and $b'$ is anticomplete to $A_2$. We next show that some  vertex in $A_2$ is complete  to $\{b, b_1\}$. Suppose no vertex in $A_2$ is complete to $\{b, b_1\}$. Then $A_2$ can be partitioned into $A_2'$ and $A_2''$ such that $b$ is anticomplete to $A_2'$ and complete to $A_2''$. Then $b_1$ is anticomplete to $A_2''$.  Thus $A_1\cup A_2'$ is anticomplete to $\{b,b', b_2\}$ and $A_2''$ is anticomplete to $\{b', b_1,b_2,  b_3\}$. Moreover,   $A_1\cup A_2'\cup\{b, b_2\}$ and $A_2''\cup \{b', b_1, b_3\}$ are independent  sets in $G$ because $G$ is $2K_2$-free.  Let $U:=V(B)\cup A_1\cup A_2$. Then  $\chi(G[U])=2$ and $\chi(G\less U)\ge \chi(G)-2$. Note that each  vertex in $G\less U$ is   adjacent to at least one vertex in  both  $D_1$  and  $D_2$.    By the minimality of $G$,  $G\less U$ contains a dominating clique minor of order $\chi(G)-2$, say $(T_3, \ldots, T_{\chi(G)})$.  Then $(G[D_1], G[D_2], T_3, \ldots, T_{\chi(G)})$ yields a  dominating clique minor of order $\chi(G)$ in $G$, a contradiction. Thus there exists a vertex, say $b_5$,  in $  A_2$ such that $b_5$ is  complete to     $\{b, b_1\} $  and anticomplete to   $\{b', b_2, b_3\} $. By symmetry, 
 there exists a vertex, say  $b_4$, in $G\less V(B)$ such that  $b_4$ is complete to     $\{b, b_3\} $  and anticomplete to   $\{b', b_1, b_2\} $. Then $b_4\ne b_5$ and so $b_4b_5\in E(G)$, else $G[\{b_1, b_5, b_3, b_4\}]=2K_2$. It follows that $G[V(B)\cup\{b_4, b_5\}]=T$.     
 \end{proof}
\noindent {\bf Claim\refstepcounter{counter}\label{C5}  \arabic{counter}.}
 $G$ contains  $C_5$, the cycle of length five,  as an induced subgraph.
 
 \begin{proof}   Suppose   $G$ is $C_5$-free.   Since $G$ is $2K_2$-free, we see that $G$ is $C_\ell$-free for each $\ell\ge6$. By \cref{t:forbidhole} and the fact that $h_d(G) < \chi(G)$,    $G$  contains an induced  $C_4$, say with vertices $v_1, v_2, v_3, v_4$ in order.  
 Define 
  \[I:=\{v\in V(G)\mid v   \text{ is anticomplete to } V(C_4)\}.\]
Then  $I$ is an independent set since   $G$ is  $2K_2$-free. Let $H:=G\less (I\cup V(C_4))$. By the choice of $I$, each vertex in $H$ is adjacent to at least one vertex on   $C_4$. Note that  $\chi(H)\ge \chi(G)-2$. By the minimality of $G$,   $H$ contains a dominating clique minor of order $\chi(G)-2$,  say $(T_3, \ldots, T_{\chi(G)})$.     By Claim~\ref{banner},   $G$ is banner-free because the graph $T$ in Figure~\ref{fig:banner} contains  $C_5$ as an induced subgraph. Thus each vertex in $H$ is adjacent to at least two vertices on $C_4$.   Let $D_1:=\{v_1, v_2\}$ and $D_2:=\{v_3, v_4\}$. If each vertex in $H$ is adjacent  to at least one vertex in both $D_1$ and  $D_2$, then   $(G[D_1], G[D_2], T_3, \ldots, T_{\chi(G)})$ yields a  dominating clique minor of order $\chi(G)$ in $G$, a contradiction.  Thus there exists a vertex $x$ in   $H$ such that $x$ is anticomplete to $D_1$ or $D_2$, say $D_2$. Then $x$ is complete to $D_1=\{v_1, v_2\}$. Let $D_1':=\{v_1, v_4\}$ and $D_2':=\{v_2, v_3\}$. Using a similar argument, there exists a vertex $y$ in  $H$ such that $y$ is anticomplete to $D_1'$ or $D_2'$, say $D_2'$.  Then $y$ is complete to $D_1'=\{v_1, v_4\}$. Thus $x\ne y$ and so $xy\in E(G)$,  else $G[\{x, v_2, y, v_4 \}]=2K_2$. It follows that $G[ \{x, v_2, v_3, v_4, y\}]=C_5$, a contradiction.  
 \end{proof}
    
   By Claim~\ref{C5}, $G$ contains     an induced $C_5$, say with vertices $v_1, v_2, v_3, v_4, v_5$ in order.  Define 
  \[I:=\{v\in V(G)\mid v   \text{ is anticomplete to } V(C_5)\}.\]
Then  $I$ is  an independent set.    Let $H:=G\less (I\cup V(C_5))$.  By the choice of $I$, each vertex in $H$ is adjacent to at least one vertex on   $C_5$.\bigskip

\noindent {\bf Claim\refstepcounter{counter}\label{nbrs}  \arabic{counter}.}
 Each vertex in $H$ is adjacent to at least four vertices on   $C_5$.
 \begin{proof} 
 Suppose there exists a vertex $x\in V(H)$ such that $x$ is adjacent to at most three vertices on $C_5$. We choose $C_5$ and $x$   such that $|N(x)\cap V(C_5)|$ is minimum. We may assume that $xv_1\in E(G)$ because $x\notin I$.   Suppose $x$ is anticomplete to $\{v_3, v_4\}$. Then $G[\{x, v_1, v_3, v_4\}]=2K_2$, a contradiction. Thus  $xv_3\in E(G)$ or  $xv_4\in E(G)$. We may assume that $xv_3\in E(G)$. Then $xv_2\notin E(G)$, else $G[\{x, v_2, v_4, v_5\}]=2K_2$ because $x$ is adjacent to at most three vertices on $C_5$.   We may further assume that $xv_4\notin E(G)$ because $x$ is adjacent to at most three vertices on $C_5$.     
 Then  $(x, v_1, v_2, v_3; v_4)$ is an induced  banner in $G$. By Claim~\ref{banner}, there exist $y,  z\in V(H)$ such that $yz\in E(G)$,  $y $ is complete to $\{x, v_3\}$ and anticomplete to $\{v_1, v_2, v_4\}$, and $z$ is complete to $\{v_2,   v_3\}$ and anticomplete to $\{x, v_1, v_4\}$. Note that $N(y)\cap V(C_5)\subseteq\{v_3, v_5\}$. By the choice of $x$,   we have $xv_5\notin E(G)$. Then $(x, v_3, v_2, v_1; v_5)$ is an induced  banner in $G$. By Claim~\ref{banner}, there exist $u,  w\in V(H)$ such that $uw\in E(G)$,  $u $ is complete to $\{v_1, v_2\}$ and anticomplete to $\{x, v_3, v_5\}$, and $w$ is complete to $\{v_1,   x\}$ and anticomplete to $\{ v_2, v_3, v_5\}$. 
 Since $G[\{ y,z, v_3, v_4, u,w, v_1,v_5\}]$  is  $2K_2$-free, we see that $G[\{y,z, u, w\}]=K_4$,  $v_4$ is complete to $\{u, w\}$  and $v_5$ is complete to $\{y, z\}$.  Let $D_1 := \{x, v_1, v_5\}$, $D_2 := \{v_2, v_3, v_4\}$, $D_3 := \{z, u\}$, and $D_4 := \{y, w\}$. It can be easily checked that $(G[D_1], G[D_2], G[D_3], G[D_4])$ yields    a dominating $K_4$ minor.  Let $S:=\{x, y,z,u,w\}$ and 
 \[A_1:=\{v \in V(H) \less S\mid v \text{  is anticomplete to } D_1\},\]
\[A_i := \{v \in V(H) \backslash (S \cup (\cup_{k=1}^{i-1} A_k ))\mid v \text{ is anticomplete to } D_i\} \text{ for each }  i\in\{2,3,4\}.\]
Let $H^*:=G\less (I\cup V(C_5)\cup S\cup (\bigcup_{k=1}^{4} A_k)$. Then $A_1, \ldots, A_4$ are pairwise disjoint and each vertex in $H^*$ is adjacent to at least one vertex in   $D_j$ for each $j\in[4]$.   We next show that $\chi(H^*)\ge \chi(G)-4$.    Note that   $A_1\cup \{z, v_4\}$ is anticomplete to $\{x, v_1\}$, $A_2\cup\{w, v_5\}$ is anticomplete to $\{v_2, v_3\}$, $A_3\cup\{x\}$ is anticomplete to $\{z, u\}$, $A_4\cup\{v_2\}$ is anticomplete to $\{y, w\}$,  and $x$ is anticomplete to $I$ because $\{x\}\cup I$ is anticomplete to $\{v_4, v_5\}$. Thus  $z$ is anticomplete to $I$ because $\{z\}\cup I$ is anticomplete to $\{x, v_1\}$. It follows that   $A_1\cup I \cup \{z, v_4\}$, $A_2\cup\{w, v_5\}$, $A_3\cup\{x\}$ and $A_4\cup\{v_2\}$ are pairwise disjoint independent sets in $G$  because  $G$ is $2K_2$-free. Thus $A_1\cup I\cup \{v_1, v_4, z\}$, $A_2\cup\{v_3, v_5, w\}$, $A_3\cup\{x, u\}$ and $A_4\cup\{v_2,  y\}$ are pairwise disjoint independent sets in $G$, and so   $G[I\cup V(C_5)\cup S \cup (\bigcup_{k=1}^{4} A_k )]$ is $4$-colorable.   This proves that  $\chi(H^*)\ge \chi(G)-4$. By the minimality of $G$, $H^*$ contains a dominating clique minor of order $\chi(G)-4$, say    $(T_5, \ldots, T_{\chi(G)})$.  Then   $(G[D_1], G[D_2], G[D_3], G[D_4], T_5, \ldots, T_{\chi(G)})$ yields a  dominating clique minor of order $\chi(G)$ in $G$, a contradiction.     
 \end{proof}
  
 By Claim~\ref{nbrs}, every vertex in $H$ is adjacent to at least four vertices on   $C_5$. By the arbitrary choice of $C_5$, we see that for any induced cycle $C$ of length five in $G$, every vertex in $V(G)\less V(C)$ is either anticomplete to $V(C)$ or adjacent to at least four vertices on $C$. Let  
 \[J:=\{v\in V(H)\mid v \text{ is complete to } V(C_5)\} \text{ and }\]
  \[Y_i:= \{v\in V(H)\mid vv_i\notin E(G)\} \text{ for each  } i\in[5].\]
Then $\{J, Y_1, \ldots, Y_5\}$ partitions $V(H)$.     Let $Y:= \bigcup_{k=1}^{5} Y_k$. \\

\noindent {\bf Claim\refstepcounter{counter}\label{IY}  \arabic{counter}.}
    $I$ is anticomplete to $Y$.
 \begin{proof}
Suppose there exist $u\in I$ and  $v \in Y$ such that $uv\in E(G)$. We may assume that $v\in Y_1$. Then $( v_5, v_1, v_2, v; u)$ induces a banner in $G$. By Claim~\ref{banner},    there   exists a vertex $w \in V(H)$ such that $w$ is complete  to $\{v, v_2\}$ and anticomplete to $\{u, v_1, v_5\}$. Thus $w$ is  adjacent to at most three vertices on $C_5$, which contradicts   Claim~\ref{nbrs}.   
\end{proof}

\noindent {\bf Claim\refstepcounter{counter}\label{Yclique}  \arabic{counter}.}
    For each $i \in [5]$, $Y_i$ is a clique of order at least two.
 
 \begin{proof}
Suppose $Y_i$ is not a clique  for some $i\in[5]$.  We may assume that $i=1$.  Let  $u,v\in Y_1$   such that $uv\notin E(G)$. Then $(u, v_4, v, v_2;v_1)$ induces a banner in $G$. By Claim~\ref{banner}, there   exists a vertex  $w\in V(H)$ such that $w$ is complete to $\{u,v_2\}$ and anticomplete to $\{v, v_1, v_4\}$. Then $w$ is adjacent to at most three vertices on $C_5$, which contradicts  Claim~\ref{nbrs}.  Thus $Y_i$  is a clique   for all $i \in [5]$. We next show that $|Y_i|\ge2$ for each $i \in [5]$. \medskip
 
   Suppose   $Y=\es$. Then $\{V(C_5), I, J\}$  partitions   $V(G)$ and $\chi(G[J])\ge \chi(G)-3$. By the minimality of $G$,  $G[J]$ contains a dominating clique minor of order $\chi(G)-3$, say    $(T_4,  \ldots, T_{\chi(G)})$.    Then  $(G[\{v_1, v_2, v_3\}], G[\{v_4\}], G[\{v_5\}],   T_4,   \ldots, T_{\chi(G)})$ yields a  dominating clique minor of order $\chi(G)$ in $G$, a contradiction.    Thus $Y\ne\es$. 
 We may assume that $Y_1\ne \es$.  Let $y_1\in Y_1$. By Claim~\ref{IY},  $G[V(C_5)\cup I\cup \{y_1\}]$ is $3$-colorable. Thus $\chi(H\less y_1)\ge \chi(G)-3$. By the minimality of $G$, $H\less y_1$ contains a dominating clique minor of order $\chi(G)-3$, say    $(  T_4,  \ldots, T_{\chi(G)})$. Suppose $Y_1=\{y_1\}$ or $Y_2=\es$. Then $v_1$ is complete to $V(H)\less \{y_1\}$ if $Y_1=\{y_1\}$ and $v_2$ is complete to $V(H)$ if $Y_2=\es$.  
     Let 
     \[D_1 := \{y_1, v_2, v_3\}, D_2 := \{v_4, v_5\},  D_3 := \{v_1\} \text{ if } Y_1=\{y_1\},\] 
      \[D_1 := \{y_1, v_1, v_5\}, D_2 := \{v_3, v_4\},  D_3 := \{v_2\} \text{ if } Y_2=\es.\]
 It is easy to check that $(G[D_1], G[D_2], G[D_3])$ yields a dominating $K_3$ minor, and   each vertex in $H\less y_1$ is adjacent to at least one vertex in  $D_j$ for all $j\in[3]$.  Thus  $(G[D_1], G[D_2], G[D_3], T_4,   \ldots, T_{\chi(G)})$ yields a  dominating clique minor of order $\chi(G)$ in $G$, a contradiction.  This proves that $|Y_1|\ge 2$ and $Y_2\ne \es$. By symmetry,  $Y_i$ is a clique of order at least two for each $i\in[5]$.  
\end{proof}

\noindent {\bf Claim\refstepcounter{counter}\label{JY}  \arabic{counter}.}
   Every vertex in $J$ is adjacent to at least $|Y_i|-1\ge1$ vertices in $Y_i$ for    all $i \in [5]$.
 \begin{proof}
 By Claim~\ref{Yclique}, $|Y_i|-1\ge1$ for    all $i \in [5]$. Suppose there exists a vertex $v\in J$ such that $v$ is not adjacent to two vertices, say $y_i, y_i'$, in $Y_i$ for some $i\in[5]$.     Since  $J$ is complete to $V(C_5)$, we see that  $G[\{v, v_i, y_i, y_i'\}]=2K_2$, a contradiction.   \end{proof}

For the remainder of the proof, all arithmetic on indices  is done modulo $5$.\\

\noindent {\bf Claim\refstepcounter{counter}\label{YiYj}  \arabic{counter}.}
    For each $i \in [5]$, $Y_i$ is complete to $Y_{i+2}\cup Y_{i-2}$. Furthermore, every vertex in $Y_i$ is adjacent to at least $|Y_{i+1}|-1$ vertices in $Y_{i+1}$ and   at least $|Y_{i-1}|-1$ vertices in $Y_{i-1}$.
 
 \begin{proof}
  Let $y_i \in Y_i$. Suppose there exists a vertex $y_{i+2} \in Y_{i+2}$ such that $y_iy_{i+2} \notin E(G)$. Then $G[\{y_i, v_{i+2}, y_{i+2}, v_i\}]=2K_2$, a contradiction. Thus $Y_i$ is complete to $Y_{i+2}$. By symmetry,  $Y_i$ is complete to $Y_{i-2}$. Next, suppose there exist  $y_{i+1}, y_{i+1}'\in Y_{i+1}$ such that $y_iy_{i+1}, y_iy_{i+1}'\notin E(G)$.  Recall that $y_iv_{i+1}\in E(G)$.   Then $G[\{y_i, v_{i+1}, y_{i+1}, y_{i+1}'\}]=2K_2$, a contradiction.  Thus each vertex in $Y_i$ is adjacent to at least $|Y_{i+1}|-1$ vertices in $Y_{i+1}$.   By symmetry, each vertex in $Y_i$ is adjacent to at least $|Y_{i-1}|-1$ vertices in $Y_{i-1}$.
\end{proof}

For each $i,j\in[5]$ with $i\ne j$, we use $ G[Y_i, Y_j]$ to denoted the bipartite graph with bipartition $\{Y_i, Y_j\}$ such that  $ E(G[Y_i, Y_j])$ consists of  all edges in $G$ with one end in $Y_i$ and the other in $Y_j$.\\

\noindent {\bf Claim\refstepcounter{counter}\label{pm}  \arabic{counter}.}
   For each $i \in [5]$,    $|Y_i|=|Y_{i+1}|$ and $ G[Y_i, Y_{i+1}]$  is  $(|Y_i|-1)$-regular.
 
\begin{proof}
  Let $y_i \in Y_i$. By Claim~\ref{YiYj},  $y_i$  is adjacent to at least $|Y_{i+1}|-1$ vertices in $Y_{i+1}$. Suppose $y_i$ is complete to $Y_{i+1}$. Let $y_{i+1} \in Y_{i+1}$ and $ y_{i-2} \in Y_{i-2}$.  Then $y_{i+1}y_{i-2}\in E(G)$ by Claim~\ref{YiYj}. Let  $S := \{v_i, v_{i+1}, v_{i-2}, y_i, y_{i+1}, y_{i-2}\}$.   Note that $I$ is anticomplete to $S$ by Claim~\ref{IY} and the choice of $I$. Then $G[S\cup I]$ is $3$-colorable  and so $\chi(G\less(S\cup I))\ge\chi(G)-3$. By the minimality of $G$, $G\less(S\cup I)$ contains a dominating clique minor of order $\chi(G)-3$, say    $(T_4,  \ldots, T_{\chi(G)})$. 
 Let $D_1 := \{y_i, v_{i+1}\}$, $D_2 := \{y_{i+1}, v_{i-2}\}$, and $D_3 := \{y_{i-2}, v_i\}$.  Since $y_i$ is complete to $Y_{i+1}$, we see that $(G[D_1], G[D_2], G[D_3])$ yields a dominating $K_3$ minor and each vertex in $Y_{i+1}\less y_{i+1}$ is adjacent to $y_i$ in $D_1$. By Claim~\ref{YiYj} and Claim~\ref{Yclique}, it can be easily checked that each vertex in $G\less(S\cup I)$ is adjacent to at least one vertex in $D_\ell$ for each $\ell\in[3]$. 
 Thus    $(G[D_1], G[D_2], G[D_3], T_4,   \ldots, T_{\chi(G)})$ yields a  dominating clique minor of order $\chi(G)$ in $G$, a contradiction.  This proves that no vertex in $Y_i$ is complete to $Y_{i+1}$. By symmetry,  no vertex in $  Y_{i+1}$ is   complete to $Y_i$. By Claim~\ref{YiYj},   $|Y_i|=|Y_{i+1}|$ and $ G[Y_i, Y_{i+1}]$ is $(|Y_i|-1)$-regular for each $i\in[5]$.   \end{proof}

By Claim~\ref{pm}, $ |Y_1|=\cdots=|Y_5|$. Let   $m:=|Y_1|$. Then  $m\ge2$ by Claim~\ref{Yclique}. For each $i\in[5]$, by Claim~\ref{pm},    $ G[Y_i, Y_{i-1}]$ and $ G[Y_i, Y_{i+1}]$ are both $(m-1)$-regular, that is,    every vertex   in $ Y_i$ has a unique non-neighbor in  $ Y_{i-1}$ and a unique non-neighbor in  $ Y_{i+1}$.  
 \\

\noindent {\bf Claim\refstepcounter{counter}\label{YJ}  \arabic{counter}.}
  For each $i\in[5]$, let $y_i\in Y_i$,  $y_{i-1}\in Y_{i-1}$ and $y_{i+1} \in Y_{i+1}$  such that $y_iy_{i-1}, y_iy_{i+1}\notin E(G)$.   Then no vertex in $J$ is anticomplete to $\{y_{i-1}, y_{i+1}\}$. 
 
\begin{proof}
Suppose  there exists a vertex $v\in J$ such that $v$ is anticomplete to $\{y_{i-1}, y_{i+1}\}$.  By Claim~\ref{YiYj}, $y_{i-1}y_{i+1}\in E(G)$. Let $C:=\{y_{i-1}, y_{i+1}, v_{i-1}, y_i, v_{i+1}\}$. Then $ G[C]=C_5$ and  $v$ is adjacent to  at most three vertices on $ G[C]$, contrary to Claim~\ref{nbrs}.    
 \end{proof}

Let $R := \{v_1, v_2, v_3, v_4\} \cup Y_1 \cup Y_2 \cup Y_3 \cup Y_4$. Then $I$ is anticomplete to $R$ by Claim~\ref{IY}.   We first show that $G[R\cup I]$ is $(2m+2)$-colorable. By Claim~\ref{YiYj}, $G[Y_1\cup Y_4\cup\{v_2, v_3\}]=K_{2m+2}$ and so $\chi(G[R\cup I])\ge 2m+2$. By Claim~\ref{YiYj} again,  $G[Y_1\cup Y_2]$ is $m$-colorable and $G[Y_3\cup Y_4]$ is  $m$-colorable. Note that $G[\{v_1, v_2, v_3, v_4\}]$ is $2$-colorable.  It follows that  $\chi(G[R\cup I])=2m+2$ and so $\chi(G\less (R\cup I)) \ge \chi(G) -( 2m + 2)$.  By the minimality of $G$, $G\less (R\cup I)$ contains a dominating clique minor of order $\chi(G) -( 2m + 2)$, say $(T_{2m+3}, \ldots, T_{\chi(G) })$.   For each $i\in [4]$, let $Y_i:=\{y^i_1, \ldots, y^i_m\}$ such that $y_\ell^2$ is anticomplete to $\{y_\ell^1,y_\ell^3\}$ for each $\ell\in[m]$. This is possible because $G[Y_2, Y_1]$ and $G[Y_2, Y_3]$ are $(m-1)$-regular bipartite graphs by Claim~\ref{pm}. By Claim~\ref{YJ}, no vertex in $J$ is anticomplete to $\{y^1_\ell, y^3_\ell\}$ for each $\ell\in [m]$.  \medskip

Let $m:=2p+r$, where $p$ is a positive integer and $r\in\{0,1\}$. Let $D_1 := \{v_2, v_3, v_4\}$ when $m$ is even,  and let $D_1 := \{v_2, v_3, y^4_m\}$ and $D_2 := \{v_4, y^2_m\}$ when $m$ is odd.  For  each $j\in[p]$,   let  $D_{1+r+j}:=\{y^2_{2j-1}, y^2_{2j}\}$ and $D_{p+1+r+j}:=\{y^4_{2j-1}, y^4_{2j}\}$.   For each $\ell\in [m]$, let  $D_{m+1+\ell}:=\{y^1_\ell, y^3_\ell\}$. Finally, let $D_{2m+2} := \{v_1\}$. Then $R=\bigcup_{k=1}^{2m+2}D_k$ and $D_1, \ldots, D_{2m+2}$ are pairwise disjoint. By Claim~\ref{YiYj} and Claim~\ref{JY}, we see that  $(G[D_1], \ldots, G[D_{2m+2}])$ yields a dominating  $K_{2m+2}$ minor. Note that $V(G\less (R\cup I))=\{v_5\} \cup Y_5 \cup J$. By Claim~\ref{YJ}, Claim~\ref{YiYj} and Claim~\ref{JY}, each vertex in  $G\less (R\cup I)$ is adjacent to at least one vertex in $D_i$ for each $i\in [2m+2]$.  Thus $(G[D_1], \ldots, G[D_{2m+2}], T_{2m+3}, \ldots, T_{\chi(G)})$ yields a dominating  clique minor of order $\chi(G)$ in $G$, a contradiction.   \medskip

This completes the proof of Theorem~\ref{t:main}.
\end{proof}
 
 \noindent{\bf Acknowledgements.} The authors Zi-Xia Song  and Thomas Tibbetts  thank Max Malik, an undergraduate student at the University of Central Florida, for helpful discussion.   Zi-Xia Song   thanks the Banff International Research Station and the organizers of  the ``New Perspectives in Colouring and Structure'' workshop, held  in Banff from September 29 - October 4, 2024, for their  kind  invitation and hospitality.

\end{document}